\documentclass{amsproc}
\usepackage{euscript}
\usepackage{cases}
\usepackage{mathrsfs}
\usepackage{bbm}
\usepackage{amssymb}
\usepackage{amsfonts,amsmath,amsxtra,mathdots,mathabx}
\usepackage{color}
\usepackage{hyperref}
\usepackage{tikz}
\usepackage{appendix,upgreek}

\allowdisplaybreaks

\DeclareFontFamily{U}{matha}{\hyphenchar\font45}
\DeclareFontShape{U}{matha}{m}{n}{
	<5> <6> <7> <8> <9> <10> gen * matha
	<10.95> matha10 <12> <14.4> <17.28> <20.74> <24.88> matha12
}{}
\DeclareSymbolFont{matha}{U}{matha}{m}{n}

\DeclareMathSymbol{\Lt}{3}{matha}{"CE}
\DeclareMathSymbol{\Gt}{3}{matha}{"CF}

\def\splus{\text{\scalebox{0.86}{$+$}}}
\def\sminus{\text{\scalebox{0.86}{$-$}}}
\def\spm{\text{\scalebox{0.86}{$\pm$}}}
\def\sasymp{\text{\scalebox{0.9}{$\asymp$}}} 

\DeclareSymbolFont{mathc}{OML}{txmi}{m}{it}
\DeclareMathSymbol{\varvv}{\mathord}{mathc}{118}
\DeclareMathSymbol{\varww}{\mathord}{mathc}{119}
\DeclareMathSymbol{\varnu}{\mathord}{mathc}{"17}

\DeclareSymbolFont{mathd}{OML}{ztmcm}{m}{it}
\DeclareMathSymbol{\varalpha}{\mathord}{mathd}{11}
\DeclareMathSymbol{\vvlambda}{\mathord}{mathd}{21}
 
\def\varlambda{\text{\raisebox{- 2 \depth}{\scalebox{0.9}[0.98]{$\vvlambda$}}\hskip -1pt}}

\def\SO{\text{\raisebox{- 2 \depth}{\scalebox{1.1}{$ \text{\usefont{U}{BOONDOX-calo}{m}{n}O} \hskip 0.5pt $}}}}

\DeclareMathSymbol{\depsilon}{\mathord}{mathd}{15}

\def\vepsilon{\upvarepsilon}

\DeclareMathSymbol{\varchi}{\mathord}{mathd}{31}

\newcommand{\BQ}{{\mathbb {Q}}}

 \newcommand{\BZ}{{\mathbb {Z}}}

\newcommand{\RC}{{\mathrm {C}}}

\newcommand{\GL}{{\mathrm {GL}}}

\newcommand{\SL}{{\mathrm {SL}}}

\newcommand{\ds}{\displaystyle}

\newcommand{\ra}{\rightarrow}

\def\-{^{-1}}

\def\mod{\mathrm{mod}\,  }

\def\nd{\mathrm{d}}

\def\SB{\text{\raisebox{- 2 \depth}{\scalebox{1.1}{$ \text{\usefont{U}{BOONDOX-calo}{m}{n}B} \hskip 0.5pt $}}}}

\def\SD{\text{\raisebox{- 2 \depth}{\scalebox{1.1}{$ \text{\usefont{U}{BOONDOX-calo}{m}{n}D} \hskip 0.5pt $}}}}

\def\SO{\text{\raisebox{- 2 \depth}{\scalebox{1.1}{$ \text{\usefont{U}{BOONDOX-calo}{m}{n}O} \hskip 0.5pt $}}}}

\def\SM{\text{\raisebox{- 2 \depth}{\scalebox{1.1}{$ \text{\usefont{U}{BOONDOX-calo}{m}{n}M} \hskip 0.5pt $}}}}

\def\SN{\text{\raisebox{- 2 \depth}{\scalebox{1.1}{$ \text{\usefont{U}{BOONDOX-calo}{m}{n}N} \hskip 0.5pt $}}}}

\def\SS{\text{\raisebox{- 2 \depth}{\scalebox{1.1}{$ \text{\usefont{U}{BOONDOX-calo}{m}{n}S}\hskip 0.5pt $}}}}

\def\ST{\text{\raisebox{- 2 \depth}{\scalebox{1.1}{$ \text{\usefont{U}{BOONDOX-calo}{m}{n}T}\hskip 1pt $}}}}

\def\SZ{\text{\raisebox{- 2 \depth}{\scalebox{1.1}{$ \text{\usefont{U}{BOONDOX-calo}{m}{n}Z}\hskip 1pt $}}}}

\def\lp {\left (}
\def\rp {\right )}

\def\shskip{\hskip 0.5 pt}
\def\snatural{\text{\scalebox{0.85}{$\natural$}}}

\def\SB{\text{\raisebox{- 2 \depth}{\scalebox{1.1}{$ \text{\usefont{U}{BOONDOX-calo}{m}{n}B} \hskip 0.5pt $}}}}

\def\SS{\text{\raisebox{- 2 \depth}{\scalebox{1.1}{$ \text{\usefont{U}{BOONDOX-calo}{m}{n}S}\hskip 0.5pt $}}}}

\def\ST{\text{\raisebox{- 2 \depth}{\scalebox{1.1}{$ \text{\usefont{U}{BOONDOX-calo}{m}{n}T}\hskip 1pt $}}}}

\def\SZ{\text{\raisebox{- 2 \depth}{\scalebox{1.1}{$ \text{\usefont{U}{BOONDOX-calo}{m}{n}Z}\hskip 1pt $}}}}

\def\SDH{\text{\raisebox{- 1 \depth}{\scalebox{1.03}{$ \text{\usefont{U}{dutchcal}{m}{n}H}  $}}}}

\newcommand{\delete}[1]{}

\theoremstyle{plain}

\newtheorem{thm}{Theorem}[section] \newtheorem{cor}[thm]{Corollary}
\newtheorem{lem}[thm]{Lemma}

\newtheorem {rem}[thm]{Remark}

\numberwithin{equation}{section}

\begin{document}

	\title[Cubic Moment of Maass-form $L$-functions]{{Proof of the Strong Ivi\'c Conjecture for the Cubic Moment of Maass-form $L$-functions}}

\begin{abstract}
	In this paper, we prove the following asymptotic formula for the spectral cubic moment of central $L$-values:
	\begin{equation*} 	
		\sum_{t_f  \leqslant T} \frac {2 L \big( \tfrac 1 2 , f \big)^3} {L(1, \mathrm{Sym}^2 f)}  + \frac {2}  {\pi}   \int_{ 0}^{T} \frac {\left| \zeta  \big(\tfrac 1 2 + it \big) \right|^{6} } { | \zeta   (1 + 2 it  )  |^2 }
		\shskip   \mathrm{d} \shskip t = T^2 P_3 (\log T) + O (T^{1+\vepsilon}) ,
	\end{equation*}
where $f$ ranges in an orthonormal basis of (even) Hecke--Maass cusp forms, and $P_3$ is a certain polynomial of degree $3$. It improves on the error term $O (T^{8/7+\vepsilon})$ in a paper of Ivi\'c and hence confirms his strong conjecture for the cubic moment. This is the first time that the (strong) moment conjecture is fully proven in a cubic case. Moreover, we establish the short-interval variant of the above asymptotic formula on intervals of length as short as $T^{\vepsilon}$. 
\end{abstract}

\author{Zhi Qi}
\address{School of Mathematical Sciences\\ Zhejiang University\\Hangzhou, 310027\\China}
\email{zhi.qi@zju.edu.cn}

\thanks{The author was supported by a grant (\#12071420) from the National Natural Science Foundation of China.}

\dedicatory{\normalsize In memory of Professor Aleksandar Ivi\'c.}

\subjclass[2010]{11F66, 11F12}
\keywords{Maass cusp forms,  $L$-functions, the cubic moment,  the moment conjecture.}

\maketitle

\section{Introduction}

\subsection{Ivi\'c's Moment Conjectures}

Let  $\SB $ be an orthonormal basis of even Hecke--Maass cusp forms $f (z)$ for $\SL_2 (\BZ)$, with  Laplacian eigenvalue $ \frac 1 4 + t_f^2$ ($t_f \geqslant 0$) and Fourier expansion 
\begin{align*}
	f  (z) = \sum_{n \neq 0} \rho_f (n) \sqrt{y} K_{it_f} (2\pi|n| y) e (n x), \qquad z= x+iy, \, y > 0.
\end{align*}     
Let $L(s, f)$ and $L (s, \mathrm{Sym}^2 f)$ be the standard and the symmetric square $L$-functions attached to $f (z)$.  Define $\omega_f$ to be the harmonic weight  
\begin{align}
	  \omega_f = \frac {|\rho_f (1)|^2} {\cosh (\pi t_f)} = \frac {2} {L(1, \mathrm{Sym}^2 f)} . 
\end{align}  Define the $k$-th moment of central $L$-values of   height $T$   as follows: 
\begin{align}\label{1eq: truncated moments}
	\SM_k  (T) = 	\sum_{t_f  \leqslant T} \omega_f L \big( \tfrac 1 2 , f \big)^k + \frac {2}  {\pi}   \int_{\shskip 0}^{T}        \frac {\left| \zeta  \big(\tfrac 1 2 + it \big) \right|^{2 k} } { | \zeta   (1 + 2 it  )  |^2 }
	\shskip   \nd \shskip t .
\end{align} 
In 2002, Ivi\'c \cite{Ivic-Moment-1} conjectured  that 
\begin{align}\label{0eq: Ivic conjecture}
	\SM_k (T) = T^2 P_{  (k^2-k)/2} (\log T) + O_{\vepsilon} (T^{1+c_k+\vepsilon}), 
\end{align}
where $P_{(k^2-k)/2} $ is a suitable polynomial of degree $   (k^2-k)/2$ whose coefficients depend on $k$, and $0 \leqslant c_k < 1$. He verified \eqref{0eq: Ivic conjecture} in the cases $k=3$ and $k= 4$ with $c_3 = 1/7$ and $c_4 = 1/3$.\footnote{For $k=4$, however, this result was first claimed in a preprint of  Kuznetsov \cite{Kuznetsov-4th-Moment} in 1999, for which Ivi\'c provided two proofs---the first is  a correction and simplification of the original proof of Kuznetsov, and the second is an elaboration of a method of  Jutila \cite{Jutila-4th-Moment}.}  Furthermore, a stronger conjecture of Ivi\'c is that $c_k = 0$; namely
\begin{align}\label{0eq: Ivic conjecture, strong}
	\SM_k (T) = T^2 P_{ (k^2-k)/2} (\log T) + O_{\vepsilon} (T^{1+ \vepsilon}).  
\end{align}
For $k = 0$, it may be considered as the (weighted) Weyl law if the odd Maass forms are included (note that the central $L$-value vanishes if the form is odd), and Xiaoqing Li \cite{XLi-Weyl} proved the currently best error bound $O (T/\log T)$. For $k = 1$ and $k = 2$, \eqref{0eq: Ivic conjecture, strong} is proven respectively by   Ivi\'c--Jutila \cite{Ivic-Jutila-Moments} and Motohashi \cite{Motohashi-JNT-Mean} with sharper error terms $O (T \log^{\vepsilon} T)$ and $ O (T \log^{6} T) $. For $k \geqslant 5$, Ivi\'c's (weak) conjecture \eqref{0eq: Ivic conjecture} is still wide open.

For other typical families of  $L$-functions, we refer the reader to the work of Conrey,  Farmer, Keating, 
Rubinstein, and Snaith \cite[\S 1.3]{Conrey-FKRS-Moments} for their moment conjectures and connection to the Random Matrix Theory. A common feature of these conjectures is the ``square-root rule''---the magnitude of the error terms is  the square root of that of the main terms. As such, one may consider the strong Ivi\'c  conjecture \eqref{0eq: Ivic conjecture, strong} as the Maass-form analogue   of these moment conjectures.\footnote{There are abundant works related to the moments of $L$-functions for holomorphic modular forms, in either the level or the weight aspect. See for example  \cite{Duke-1995,IS-Siegel,KM-Analytic-Rank,VanderKam-Rank,BF-Moments} for $k =1, 2$, \cite{CI-Cubic,Peng-Weight,Petrow-Cubic,Young-Cubic,PY-Cubic,PY-Weyl2,Frolenkov-Cubic}   for $k=3$,  \cite{DFI-2,DFI-3,KMV-4th-Moment,Blomer-Reciprocity,Young-5th-Moment,Khan-5th} for $k = 4, 5$, but some of which, especially for $k=3, 4, 5$, are   concerned with bounds instead of asymptotics.}

In this paper, we prove Ivi\'c's strong conjecture for $k=3$. This is the first instance of the  moment conjecture that is fully proven in the cubic case $k=3$.\footnote{It is interesting to note the controversy for the cubic moment of quadratic Dirichlet $L$-functions:  \cite{Conrey-FKRS-Moments} conjectured that the error term  should have exponent $  1 / 2 $, but \cite{D-Goldfeld-H-Multiple-Dirichlet,Zhang-Dirichlet-3/4} suggested the existence of a lower-order term of exponent $3/4$. See also \cite{Young-Cubic-Moment-Dirichlet} and \cite{A-R-Dirichlet-Experiments}. }

\subsection{Main Results}

More generally, we   consider the cubic moment on short intervals.   For $ T^{\vepsilon} \leqslant M \leqslant T/3 $ define
\begin{align}
\SM_3 (T, M) =	 \sum_{  |t_f - T| \shskip \leqslant M}  \omega_f  L \big( \tfrac 1 2 , f \big)^3 + \frac 2 {\pi} \int_{\, T-M}^{T+M}        \frac {\left| \zeta \big(\tfrac 1 2 + it \big) \right|^{6} } { | \zeta  (1 + 2 it  )  |^2 } \nd \shskip t ,
\end{align}
and its smoothed variant 
\begin{align}\label{1eq: defn of N flat}
	\SM_3^{\snatural} (T, M) =	\sum_{f \in  \SB}  \omega_f L \big( \tfrac 1 2 , f \big)^3 k_{T, M}  (t_f)  + \frac 1 {\pi} \int_{-\infty}^{\infty}    \frac {\left| \zeta \big(\tfrac 1 2 + it \big) \right|^{6} } { | \zeta  (1 + 2 it  )  |^2 }
	k_{T, M} (t) \shskip   \nd \shskip t ,
\end{align}
where 
\begin{align}\label{1eq: defn k(nu)}
	  k_{T, M} (t)  =  e^{- (t  - T)^2 / M^2} + e^{-(t + T)^2 / M^2}  . 
\end{align}

A well-known result for the cubic moment of $L\big( \frac 12, f \big)$ on short intervals is the following average Lindel\"of bound of Ivi\'c \cite{Ivic-t-aspect}:
\begin{align}\label{1eq: Ivic, short interval}
	\sum_{  |t_f - T| \shskip \leqslant M}     L \big( \tfrac 1 2 , f \big)^3 \Lt_{\vepsilon} M T^{1+\vepsilon}  , 
\end{align}
By the non-negativity of $ L \big( \tfrac 1 2 , f \big)$, its Weyl-type subconvex bound follows immediately from \eqref{1eq: Ivic, short interval}.

As shown in the next lemma, $\SM_k (T, M)$ and $\SM_k^{ \snatural} (T, M)$ are connected via  an   averaging process for the $T$-parameter. 

\begin{lem}\label{lem: average}
	For  $ T^{\vepsilon} \leqslant M^{1+\vepsilon} \leqslant  H \leqslant T /3 $ we have 
	\begin{align}\label{1eq: M = int M}
		 \SM_3 (T, H) = \frac 1 {\sqrt{\pi} M} \int_{T-H}^{T+H} \SM_3^{\snatural} (K, M) \nd K + O_{\vepsilon}  (M T^{1+\vepsilon} ). 
	\end{align}
\end{lem}

This is a simple application of Lemma 5.3 and 5.4 in \cite{Qi-Liu-Moments} (the short-interval variant of the unsmoothing process in Ivi\'c--Jutila \cite[\S 3]{Ivic-Jutila-Moments}), along with Ivi\'c's bound \eqref{1eq: Ivic, short interval}, the non-negativity of $ L \big( \tfrac 1 2 , f \big)$, and also the Weyl subconvex bound for $\zeta \big(\frac 1 2 + it\big)$. 

Our main results  are the following asymptotic formulae for $\SM_3^{\snatural} (T, M)$, $\SM_3 (T, H)$, and $\SM_3 (T)$. 

\begin{thm}\label{thm: T and M} For any $T^{\vepsilon} \leqslant M \leqslant T^{1-\vepsilon}$ we have 
	\begin{align}\label{1eq: smooth, asymptotic}
		\SM_3^{\snatural} (T, M) = \sqrt{\pi} M T P^{\snatural}_3 (\log T) + O_{\vepsilon} (T^{1+\vepsilon}), 
	\end{align}
where $ P^{\snatural}_3 $ is an explicit cubic polynomial. 
\end{thm}

\begin{thm}\label{thm: T and H}
	For any $T^{\vepsilon} \leqslant H \leqslant T/3$ we have 
	\begin{align}\label{1eq: asymptotic}
		\SM_3  (T, H) = \int_{T-H}^{T+H}  K P^{\snatural}_3 (\log K) \nd K + O_{\vepsilon} (T^{1+\vepsilon}). 
	\end{align} 
\end{thm}

\begin{cor}\label{cor: Ivic}
	We have
	\begin{align}
	\SM_3 (T) = T^2 P_{ 3} (\log T) + O_{\vepsilon} (T^{1 +\vepsilon}), 
	\end{align}
for an explicit cubic polynomial  $ P_3 $.
\end{cor}

Although the error terms for $\SM_3^{\snatural} (T, M) $ and $\SM_3  (T, H) $ in Theorem \ref{thm: T and M} and \ref{thm: T and H} are of the same strength, the proof of the latter requires more refined analysis.

\subsection{Backgrounds and Our Method} 

For the twisted second moment of central $L$-values for Maass forms, an explicit formula of Kuznetsov--Motohashi   (\cite{Kuznetsov-Motohashi-formula,Motohashi-JNT-Mean,Motohashi-Riemann}) is particularly useful. It was used by the ``Troika", Motohashi, Ivi\'c, and Jutila \cite{Motohashi-JNT-Mean,Motohashi-Riemann,Ivic-t-aspect,Ivic-Moment-1,Jutila-4th-Moment} to study the second, third, and fourth moments, and recently by \cite{SH-Liu-Maass,BHS-Maass} to obtain lower bounds for the non-vanishing proportion. 

In our recent work \cite{Qi-Liu-Moments},  the results in \cite{BHS-Maass} were recovered and  extended to imaginary quadratic fields by the Kuznetsov--Vorono\"i approach rather than the Kuznetsov\allowbreak--Motohashi formula (such a formula is currently not available over imaginary quadratic fields). It naturally drives us to  revisit the problem of Ivi\'c for the cubic moment from this perspective.    

The study of cubic moment for $\GL_2$ via the (Petersson--Kuznetsov) trace formula and the (triple Poisson) summation formula was initiated in  the groundbreaking work of Conrey and Iwaniec \cite{CI-Cubic}. Their focus is on the $q$-aspect, and, in the most simplified setting, their result in the Maass-form case reads as follows: 
\begin{align*}
	\sum_{t_f  \leqslant T}   L \big( \tfrac 1 2 , f \times \chi_q \big)^3 +    \int_{\shskip 0}^{T}          {\left| L  \big(\tfrac 1 2 + it, \chi_q \big) \right|^{6} }  
	\shskip   \nd \shskip t \Lt_{T, \vepsilon} q^{1+\vepsilon}, 
\end{align*}
where $\chi_q$ is the quadratic character of square-free conductor $q$.  

For the spectral aspect, novel ideas were introduced by Xiaoqing Li \cite{XLi2011} in her study of the first moment for  $\GL_3 \times \GL_2$,  and by Young \cite{Young-Cubic} for his hybrid version of Conrey and Iwaniec's results: the former used the (Vorono\"i) summation formula twice, while the latter used the hybrid large sieve after the (triple Poisson) summation formula.  
Later, Nunes \cite{Nunes-GL3} used Young's idea to improve Xiaoqing Li's subconvexity bounds for $\GL_3$, replacing the second Vorono\"i by the large sieve, and the author \cite{Qi-GL(3)} extended his results to arbitrary number fields.  

Our approach, simply speaking,  combines those of Xiaoqing Li and Young: After Kuznetsov, use Vorono\"i+Poisson, the second Vorono\"i, and finally the large sieve. The use of Vorono\"i+Poisson instead of the triple Poisson is beneficial for us to see the main term (see \cite{Qi-Liu-Moments} and the opening discussions in \S \ref{sec: setup}).  
Similar ideas with the second Vorono\"i replaced by the $\GL_3$ functional equation were used in the recent joint work \cite{LN-Qi-GL(3)} to get strong subconvexity bounds for $\GL_3$.

\subsection{Remarks}  
 

The study of the $2k$-th moment   of      $\zeta \big(\frac 1 2 +it\big)$  
is quite a fascinating story  on its own. More explicitly, the moment conjecture for $\zeta \big(\frac 1 2 +it\big)$ reads 
\begin{align*}
 	\int_{0}^T \left| \zeta  \big(\tfrac 1 2 + it \big) \right|^{2 k} \nd \shskip t = T P_{k^2} (\log T) + E_k (T),
\end{align*}
with $P_{k^2} $ a certain polynomial of degree $k^2$ and 
\begin{align*}
	E_k (T) = O_{\vepsilon} (T^{1/2+\vepsilon}). 
\end{align*} See  \cite[\S 1.3]{Conrey-FKRS-Moments}. Along with the standard bound $ 1/|\zeta(1+it)| \Lt \log t $ ($t > 3$) (see \cite[Theorem 5.17]{Titchmarsh-Riemann}), the moment conjecture for $\zeta \big(\frac 1 2 +it\big)$ implies that  in the asymptotic formula given by \eqref{1eq: truncated moments} and \eqref{0eq: Ivic conjecture, strong} the  integral  could be absorbed into the error term $ O (T^{1+\vepsilon})$. However, this can be done unconditionally only for $k = 1, 2$.\footnote{The currently best estimates for $E_1(T)$ and $E_2(T)$ are $O \big(T^{ {1515} / {4816}+\vepsilon} \big)$ and $O \big(T^{  2/3} \log^C T\big)$ ($C$ is an effective constant) due to   Bourgain--Watt \cite{BW-Riemann-2} and  Ivi\'c--Motohashi \cite{IM-4th-Moment}. For $k=6$, Heath-Brown \cite{Heath-Brown-12th} proved that the 12th moment integral is $O \big(T^2 \log^{17} T \big)$. For $k = 3$ and $4$, the best bounds are $ O \big(T^{  5 / 4} \log^{  {37} / 4} T\big)$ and $ O \big(T^{  3 / 2} \log^{  {21} / 2} T\big)$ (see \cite{Ivic-Moment-1}), which are trivial consequences of Heath-Brown's bound and the H\"older (or Cauchy--Schwarz) inequality.}

Our analysis suggests that the error term for the smooth {\it cubic} moment $\SM_3^{\snatural} (T, M) $ is connected to the {\it fourth} moment of $\zeta \big(\frac 1 2 + it\big)$ up to the height $U = T/M$ (see \S \ref{sec: apply large sieve}), for which the bound $O (U^{1+\vepsilon})$ is known as indicated above, so Theorem \ref{thm: T and M} is probably optimal. It is interesting to see whether a lower order term can be extracted from the error term. 

The heuristic connection above should be regarded as the inverse of the Motohashi formula (see \cite{Motohashi-Riemann-4th,Motohashi-Riemann}), and it can be realized in explicit terms as the spectral moment formula of Chung-Hang  Kwan  \cite{Kwan} or the sepctral reciprocity formula in the on-going work of Humphries and Khan.  However, in the $q$-aspect,   this Motohashi-type connection was already visible in the work of Conrey and Iwaniec \cite{CI-Cubic}, and it has been implemented by many successors \cite{Michel-Venkatesh-GL2,Petrow-Cubic,PY-Weyl2,PY-Weyl3,Nelson-Eisenstein,Wu-Motohashi,B-F-Wu-Weyl}. 

Finally, we remark that  there seems to be substantial analytic obstacles (see   Remark \ref{rem: obstacle} and \ref{rem: no saving}) that prevent  us from obtaining  asymptotic over any number field other than $\BQ$.

\section*{Acknowledgments}
The author thanks Peter Humphries, Yongxiao Lin, Ramon M. Nunes for helpful discussions.

\section{Preliminaries} 

\subsection{Bessel Kernel} Bessel functions (as in \cite{Watson}) arise in both the Kuznetsov trace formula and the Vorono\"i summation formula. To make our exposition succinct, we introduce the Bessel kernel $B_s (x)$ defined as follows
\begin{equation}
\begin{aligned}
		&B_{s} (x) \hskip -1pt  = \hskip -1pt \frac {\pi} {\sin (\pi s) } \big( J_{-2 s} (4 \pi \hskip -1pt \sqrt {x }) \hskip -1pt - \hskip -1pt J_{2 s} (4 \pi \hskip -1pt \sqrt {x })   \big), \hskip -3pt \quad B_{s} (-x )    \hskip -1pt
	=  \hskip -1pt {4 \cos (\pi s)}    K_{2 s} (4 \pi \hskip -1pt \sqrt {x }) ,
\end{aligned}
\end{equation} 
for real $x > 0$ and complex $s$. 

For  $t$ real, the Bessel kernel $B_{it} (x)$  appears in a certain Bessel integral $\SDH (x)$ in Kuznetsov which is well understood by the works of \cite{Ivic-Jutila-Moments,XLi2011,Young-Cubic,Qi-Liu-LLZ}. Moreover, the Bessel kernel  $B_0 (x)$ arises in the Hankel transform in Vorono\"i, and the following formulae will be crucial in our analysis:
\begin{align} 
	\label{4eq: asymptotic, R+}		& B_0 (x) =   \sum_{ \pm} \frac {e (\pm (2 \sqrt{x} + 1/8))} {  \ds \sqrt[4]{x } } W_0 (\pm \sqrt{  x}) , \qquad B_0 (-x) = O  \bigg( \frac {\exp (-4\pi \sqrt{x})} {\ds \sqrt[4]{x  }} \bigg),  
\end{align}
for $x > 1$, in which  $x^j  W_0^{(j)} (x)  \Lt_{j} 1$ (see \cite[\S 6]{Qi-Liu-Moments}).

\subsection{Kuznetsov Trace Formula}

Let $\SB$ be an orthonormal basis of even Hecke--Maass forms for $\SL_2 (\BZ)$. For $f (z) \in \SB$ let  $\frac 1 4 + t_f^2$ ($t_f \geqslant 0$) be its Laplacian eigenvalue, $\lambda_f (n)$ ($n \geqslant 1$) be its Hecke eigenvalues, and $\rho_f (n)$ ($n \neq 0$) be its Fourier coefficients. By definition, $f (z)$ is even in the sense that $ f (- \widebar{z}) =    f (z)$, so that $\rho_f (-n) = \rho_f (n) $. It is known that $\rho_f (\pm n) = \rho_f (1) \lambda_f (n)$. 

Now we recall the Kuznetsov trace formula for even Maass forms. See \cite[\S 2]{XLi2011} and \cite[\S 3.3]{Qi-Liu-Moments}. 

\begin{lem} \label{lem: Kuznetsov}
	Let $h(t)$ be an even test function such that
\begin{itemize}
	\item [(1)] $h(t)$ is holomorphic in $|\mathrm{Im} (t) |\leqslant \frac 1 2 +\vepsilon$, and
	\item [(2)] $h(t) \Lt (|t|+1)^{-2-\vepsilon}$ in the above strip. 
\end{itemize}
Then for   $n_1, n_2 \geqslant 1$ we have 
\begin{align}
\begin{aligned}
	  \sum_{f \in \SB} h (t_f) \omega_f \lambda_f (n_1) \lambda_f (n_2) & + \frac 1 {4\pi} \int_{-\infty}^{\infty} h(t) \omega (t) \tau_{it} (n_1) \tau_{it} (n_2) \nd \shskip t \\
	& =  \delta_{n_1, n_2} \SDH +   \sum_{ \pm} \sum_{c=1}^{\infty} \frac {S (n_1, \pm n_2; c)} {c}  \SDH \bigg( \hskip -1pt \pm \frac {n_1 n_2} {c^2} \bigg), 
\end{aligned}
\end{align}
where $\delta_{n_1, n_2}$ is   Kronecker's $\delta$-symbol, 
\begin{align}
	\tau_{s} (n) = \tau_{-s} (n) = \sum_{ab=n} (a/b)^s, 
\end{align}
\begin{align}
	\omega_f = \frac {|\rho_f (1)|^2} {\cosh (\pi t_f)} = \frac {2} {L(1, \mathrm{Sym^2} f)},   \qquad \omega (t) = \frac {4 } {|\zeta (1+2it)|^2 }, 
\end{align}
and
\begin{align}\label{2eq: integral H}
	\SDH  = \frac {1} {2 \pi^2} \int_{-\infty}^{\infty} h(t)    {  \tanh (\pi t) t \shskip \nd \shskip t }    ,  \qquad
	\SDH (x) = \frac 1 {2\pi^2} \int_{-\infty}^{\infty} h(t) B_{it} (x)   \tanh (\pi t) t  \shskip \nd \shskip t .
\end{align} 
\end{lem}

\subsection{Poisson and Vorono\"i Summation Formulae}

The following Poisson formula is essentially a special case of \cite[(4.25)]{IK}. 
\begin{lem}\label{lem: Poisson} Let $\varww \in C_c^{\infty} (-\infty, \infty)$. Let $a, c $ be integers with $c \geqslant 1$. Then
	\begin{align}
		\sum_{n } e \lp - \frac {a n} c \rp \varww (n) & =   \sum_{m \shskip \equiv \shskip a (\mod c) }  \widehat {\varww} \Big(  \frac {m} {c} \Big)  ,
	\end{align}
where $ \widehat {\varww} $ is the Fourier transform of $\varww$ defined by
\begin{align}
\widehat {\varww} (y) =	\int_{-\infty}^{\infty} \varww (x) e  ( - x y)   \nd x .
\end{align}
\end{lem}

Note that there is no zero frequency in the case that $c > 1$ and $(a, c) = 1$. 

The following Vorono\"i   formula for the  divisor function $\tau (n) = \tau_0 (n)$ is from \cite[(4.49)]{IK}. Note that  $B_0 (x)  =  - 2\pi   Y_{0} (4 \pi \sqrt {x }) $  for $x > 0$ by \cite[3.54 (1)]{Watson}. 

\begin{lem}\label{lem: Voronoi} Let $\varww \in C_c^{\infty} (0, \infty)$. Let $a, \widebar{a},  c $ be integers with $c \geqslant 1$ and $a \widebar{a} \equiv 1 (\mod c)$. Then
	\begin{align}\label{app: Voronoi}
		\begin{aligned}
		c	\sum_{n } \tau (n) e \hskip -1pt  \lp -  \frac {a n} c \rp \hskip -1pt \varww (n)   =   2( \gamma   -    \log c) \widetilde{\varww}_0 (0)   +   \widetilde{\varww}_0' (0)   +  \hskip -1pt  \sum_{m \neq 0} \hskip -1pt   \tau (m) e \Big(   \frac {\widebar{a} m} c \Big)   \widetilde{\varww}_{0} \Big(\frac {m} {c^2} \Big), 
		\end{aligned}
	\end{align}
where $\gamma$ is Euler's constant, 
\begin{align}\label{3eq: Mellin}
	\widetilde{\varww}_0 (0) =	 \int_{0 }^{\infty}  \varww (x)   \nd x , \qquad \widetilde{\varww}_0' (0) =	 \int_{0 }^{\infty}  \varww (x)    \log x \shskip \nd x ,
\end{align}
and $ \widetilde{\varww}_{0}  $ is the
 Hankel transform of $ \varww $ {\rm(}with kernel $B_0${\rm)} defined by
\begin{align}\label{3eq: Hankel, global}
	\widetilde{\varww}_{0} (y) =    \int_{0}^{\infty}  \varww (x) B_{0}    ( x y)   \nd x, 
\end{align} 
for real $y \neq 0$. 
\end{lem}

\subsection{Approximate Functional Equations} 
	According to \cite[\S 4]{Qi-Liu-Moments}, with slightly altered notation, we have the following approximate functional equations:
\begin{equation}
	\label{5eq: AFE, 1} 
	L \big(\tfrac 1 2,    f \big)   =  2 \sum_{n  =1}^{\infty}  \frac {   \lambda_f  (n  )  } {\sqrt{n}  }     V_1   (     n    ; t_f   )   , \qquad L \big(\tfrac 1 2,    f \big)^2    =   2   \sum_{n  =1}^{\infty}    \frac {   \lambda_f  (n  ) \tau (n ) } { \sqrt{n}  }     V_2   ( n ; t_f   ),    
\end{equation}
and similarly
\begin{equation}
	\label{5eq: AFE, 2} 
	  \begin{split}
	  	& \left| \zeta \big(\tfrac 1 2 + it \big) \right|^2     =  2 \sum_{n  =1}^{\infty}  \frac {   \tau_{it}  (n  )  } {\sqrt{n}  }     V_1   ( n ; t   ) + O  \big( e^{-t^2/2} \big) , \\
	  	& \left| \zeta \big(\tfrac 1 2 + it \big) \right|^4     =   2   \sum_{n  =1}^{\infty}   \frac {   \tau_{it}  (n  ) \tau (n ) } { \sqrt{n} }     V_2   ( n ; t  )    +   O  \big( e^{-t^2  } \big) , 
	  \end{split}
\end{equation}
with 
\begin{equation}\label{5eq: def of V1 (y, t)} 
	V_1 (y; t ) \hskip -1pt = \hskip -1pt \frac 1  {2 \pi i} \hskip -1pt  \int_{(3)} \hskip -2pt 
	\delta (v, t) e^{v^2  }   y^{ -  v} \frac { \nd   v } {v} , \quad \hskip -1pt
	V_2 (y; t ) \hskip -1pt = \hskip -1pt \frac 1  {2 \pi i} \hskip -1pt \int_{(3)} \hskip -2pt \zeta (1 \hskip -1pt + \hskip -1pt 2v) 
	\delta (v, t)^2  e^{ 2 v^2  } y^{ - v} \frac { \nd v } {v} ,   
\end{equation}
for $y > 0$, where 
\begin{align}\label{4eq: def G}
	\delta (v, t) = \frac {\gamma \big(\frac 1 2 + v , t \big)  }  {\gamma \big(\frac 1 2, t   \big)   }     , 
\end{align}   
and 
	\begin{equation}\label{4eq: defn of gamma (s, f)}
		\gamma  (s, t) =   \pi ^{-   s } \Gamma \bigg(\frac { s- i t } 2 \bigg)   \Gamma \bigg(\frac { s+ i t } 2 \bigg) .
	\end{equation}  


\begin{lem}\label{lem: afq}  
For real $t$ define 
\begin{align*}
	\RC  (t) =        \sqrt{\frac 1 4 +  t^2}   .
\end{align*} 	Let  $ U  > 1 $. We have
	\begin{align}\label{2eq: derivatives for V(y, t), 1}
		V_1 (y; t ) \Lt_{ A } 
		\bigg(  1 + \frac {y} { \RC  (t)  } \bigg)^{-A} , \quad V_2 (y; t ) \Lt_{ A  } 
		\bigg(  1 + \frac {y} {\RC  (t)^{2 } } \bigg)^{-A},
	\end{align}   
	\begin{align}
		\label{1eq: approx of V1}
		\begin{split}
			V_1  (y; t) & = \frac 1 {2   \pi i } \int_{ \vepsilon - i U}^{\vepsilon + i U}  \delta (v, t)  e^{v^2  }  y^{ - v}   \frac {\nd v} {v} + O_{\vepsilon } \bigg( \frac {\RC  (t)^{  \vepsilon} } {y^{ \vepsilon} e^{U^2 / 2} } \bigg),  \\ 
		V_2  (y; t) & = \frac 1 {2   \pi i } \int_{ \vepsilon - i U}^{\vepsilon + i U}  \delta (v, t)^2 \zeta (1+2v) e^{2 v^2  }   y^{ - v}   \frac {\nd v} {v} + O_{\vepsilon } \bigg( \frac {\RC  (t)^{  \vepsilon} } {y^{ \vepsilon} e^{U^2   } } \bigg) ,
		\end{split} 
	\end{align}   
and, furthermore, if  $ |t| \Gt U^2 $ and $v$ is on the contour, then we may write $\delta (v, t) = (t/2\pi)^{v} (1 + \delta^{\flat} (v, t) )$ so that 
\begin{align}\label{2eq: delta(t)=} 
	 \frac {\partial^i \delta^{\flat} (v, t) } {\partial t^i } \Lt_{i } \frac {U^2} {|t|^{i+1}}  .
\end{align}  
Finally, for $1 \leqslant y <  \RC (t)^{2}$ we have 
\begin{align}\label{4eq: asymptotic for V2}
	V_2 (y; t) = \gamma  +  \psi_1 (t)  -  \log   \sqrt{y}        +  O_{A }   \lp \lp \frac {y}  {\RC  (t)^{2 }} \rp \hskip -8.5 pt {\phantom{\Big)}}^{A } \rp ,
\end{align} 
with $\psi_1 (t) = \lp \partial \log \gamma (s, t) /\partial s \rp|_{s=1/2}$. 
\end{lem}

\begin{proof} 
	\eqref{2eq: derivatives for V(y, t), 1} and \eqref{1eq: approx of V1} are essentially  from  \cite[Proposition 5.4]{IK} and \cite[Lemma 1]{Blomer}; see also \cite[Lemma 4.1 (1)]{Qi-Liu-Moments}. \eqref{2eq: delta(t)=}  follows readily from Stirling's formulae for $\log \Gamma$ and its derivatives (see for example \cite[\S \S 1.1, 1.2]{MO-Formulas}). \eqref{4eq: asymptotic for V2} is from \cite[Lemma 4.1 (2)]{Qi-Liu-Moments}.
\end{proof}

\subsection{The Large Sieve} \label{sec: large sieve}

The following   is  Gallagher's large sieve inequality  \cite[Theorem 2]{Gallagher-LS} in the case $q=1$.

\begin{lem}\label{lem: large sieve}
	Let $a_n$ be a sequence of complex numbers.	For $T > 1$ we have
	\begin{align}
		\int_{-T}^T \raisebox{- 0.15 \depth}{{$\bigg|$}}\sum_{n} a_n n^{it} \raisebox{- 0.15 \depth}{$\bigg|$}^2 d t \Lt \sum_{n} (T + n) |a_n|^2, 
	\end{align}
	provided that the sum of $|a_n|$ is bounded. 
\end{lem}

\section{Refined Analysis for the Bessel Integral} 

Subsequently, we shall always let $U = \log T$. For  $v_1, v_2 \in [\vepsilon - i U, \vepsilon + i U]$ define 
\begin{align}
	h (t; v_1, v_2) = 
	k  (t) \delta (v_1, t) \delta (v_2, t)^2,
\end{align}
with
\begin{align}\label{3eq: defn k(nu)}
	k (t)  =  e^{- (t  - T)^2 / M^2} + e^{-(t + T)^2 / M^2}  . 
\end{align}
Since $v_1$ and $v_2$ are inessential to our analysis, we shall simply write  $h (t) = h  (t; v_1, v_2)$ and let $\SDH (x)  $  be its associated Bessel integral as in \eqref{2eq: integral H}.  

For the analysis of $\SDH (\pm x^2) $, the most important are certain integral representations. The reader is referred to  \cite[\S \S 4, 5]{XLi2011}, \cite[\S 7]{Young-Cubic}, \cite[Appendix]{Qi-Liu-LLZ}, and \cite[\S 8.1]{Qi-GL(3)} for more details, and  also
 \cite[\S 7]{Qi-Liu-Moments} for a summary (although $h(t)$ may vary in different settings). 

	For the proof of Theorem {\rm\ref{thm: T and H}}, we shall change $T$ into $K$ as in Lemma \ref{lem: average}, but keep   $U = \log T$ for the range of $v_1, v_2$, so that $K$ and $v_1, v_2$ vary independently. Moreover, in order to perform the $K$-integration effectively, we wish to  refine  (the weight function in) the integral representations.

\begin{lem}\label{lem: H(x), |z|>1}
We may write 	 $ \SDH (x) = \SDH_{ \splus }  (x) + \SDH_{\sminus}  (x) + O_A   (T^{-A} ) $ for $|x| >  1$, with 
	\begin{equation}\label{8eq: H+natural}
		\SDH_{\spm}   (  x^2) =     {MT^{1+v }}   
		\int_{- M^{\vepsilon} / M}^{M^{\vepsilon}/ M}   g (    {   M r} )  e( Tr / \pi \mp 2 x \cosh r  ) \nd r,
	\end{equation} 
	\begin{equation}\label{8eq: H-natural}
		\SDH_{  \spm}   (- x^2) =    {MT^{1+v }}   
		\int_{- M^{\vepsilon} / M}^{M^{\vepsilon}/ M}   g (    {   M r} )  e( Tr / \pi \pm 2 x \sinh r  ) \nd r,
	\end{equation}
	for $x > 1$, where $v = v_1 +2v_2$, and $g (r)$ is a Schwartz function of the form  
	\begin{align}\label{3eq: g = g0+gflat}
		g (r ) = g^{\oldstylenums{0}} (r) + \frac {M U} T   g^{\flat} (r) ,
	\end{align}
with 
\begin{align}\label{3eq: g0}
	g^{\oldstylenums{0}} (r) = \frac {2^{1-v}} {\pi^{3/2+v}} e^{- r^2} , 
\end{align}
\begin{align}\label{3eq: g flat}
	(\nd /\nd r)^i g^{\flat} (r)  \Lt_{i, A, \vepsilon}  (1 +   {|r|}   )^{-A}.   
\end{align}
\end{lem}

 Lemma \ref{lem: H(x), |z|>1}, especially \eqref{3eq: g = g0+gflat},  can be proven by analyzing the arguments in \cite[A.2, A.4]{Qi-Liu-LLZ} more carefully with the aid of \eqref{2eq: delta(t)=} (see also the asymptotic analysis in \cite[\S 5]{Young-Cubic}).  
 
 It is important that neither the definition in \eqref{3eq: g0} nor the implied constants in \eqref{3eq: g flat}   depend on $T$ or $M$.  Evidently, the derivatives of $g(r)$ also satisfy \eqref{3eq: g flat}.
 
According to \eqref{3eq: g = g0+gflat}, for $|x| > 1$ we may write
\begin{align}\label{3eq: H + H0}
	\SDH  (x) = \SDH^{\shskip\oldstylenums{0}} (x) + \frac {M U} {T} \SDH^{\shskip\flat} (x) + O_A  (T^{-A}) ,
\end{align}
where $\SDH^{\shskip\oldstylenums{0}}  = \SDH_{\splus}^{\shskip\oldstylenums{0}}  + \SDH_{\sminus}^{\shskip\oldstylenums{0}}   $ and  $\SDH^{\shskip\flat}   = \SDH_{\splus}^{\shskip\flat}   + \SDH_{\sminus}^{\shskip\flat}  $, with  $\SDH_{\spm}^{\shskip \oldstylenums{0}}  $ and  $\SDH_{\spm}^{\shskip\flat}  $ defined in the same manner by \eqref{8eq: H+natural} and \eqref{8eq: H-natural}. 

Moreover, the following lemmas tell us the ranges that we need to focus on. 

\begin{lem}\label{lem: x small}
For $|x| \leqslant 1$ we have	$ \SDH (x) \Lt_A  MT^{1-2A} \sqrt{|x|} $.
\end{lem}

\begin{lem}\label{lem: H small}
For $|x| > 1$	we have $ \SDH  (x ) = O_A  (T^{-A})$ unless $|x| \Gt T^2$. To be precise,  $ \SDH  (x^2)   $ and  $ \SDH  (-x^2)  $ are negligibly small unless $x > M^{1-\vepsilon} T$ and $x \asymp T$ respectively. 
\end{lem}

Of course,  Lemma \ref{lem: H small}  also holds if $\SDH$ were replaced by  $\SDH^{\shskip \oldstylenums{0}} $ or  $\SDH^{\shskip\flat} $. 

\section{Refined Analysis for the Fourier--Hankel Transform}\label{sec: analysis Hankel}

In this section, we use the analysis in the work of Young \cite{Young-Cubic} (and the author \cite{Qi-GL(3)}) to study the Fourier--Hankel transform of the Bessel integral. 

Let $\varww_1  (x), \varww_2 (x) \in C_c^{\infty} (0, \infty)$ be such that $ \varww_1^{(i)} (x), \varww_2^{(i)} (x)  \Lt_{i} U^{i} $. For  $\varLambda \Gt T^2$ define 
\begin{align}\label{11eq: defn of w (x, Lmabda), R}
	\varww^{\spm} (x_1, x_2 ; \varLambda ) = \varww_1 (x_1) \varww_2 (x_2) \SDH  ( \pm \varLambda x_1 x_2 ) . 
\end{align}
The Fourier--Hankel integral transform arising after Vorono\"i and Poisson will be of the   form
\begin{align}\label{3eq: defn of Fourier-Hankel}
	\widehat{\widetilde{\varww}}^{\spm} (y_1, y_2; \varLambda) = \int_0^{\infty} \int_0^{\infty} \varww^{\spm} (x_1, x_2 ; \varLambda ) e (-x_1 y_1) B_0 (x_2 y_2) \nd x_1 \nd x_2. 
\end{align}
Define $ \widehat{ \widetilde{\varww}}^{\oldstylenums{0}\,\spm}_{} (y_1, y_2; \varLambda) $ and $\widehat{\widetilde{\varww}}^{\flat \, \spm} (y_1, y_2; \varLambda)$ in the same way.

To apply the analysis of Young, we reformulate \eqref{3eq: defn of Fourier-Hankel} in the following way (with $x=x_1 x_2$)
\begin{align}
	\widehat{\widetilde{\varww}}^{\spm} (y_1, y_2; \varLambda) = \int_0^{\infty}  \SDH  ( \pm \varLambda x )  J (x; y_1, y_2) \nd x ,
\end{align}
with Fourier--Bessel integral  kernel
\begin{align}
J (x; y_1, y_2) =	\int_0^{\infty} \varww_1 (x_1) \varww_2 (x/x_1) e (-x_1 y_1) B_0 ( x y_2 / x_1) \frac {\nd x_1} {x_1} . 
\end{align}
The next lemma provides a description of the asymptotic of $J (x; y_1, y_2)$ which is similar to those in \cite[Lemma 6.5]{Young-Cubic} and \cite[Lemma 4.12]{Qi-GL(3)} as expected.

\begin{lem}\label{lem: asymptotic of J}
Assume that $ |y_2| > T^{\vepsilon}$. Then $J (x; y_1, y_2) = O (T^{-A})$ unless $y_2 > T^{\vepsilon}$ is positive and
\begin{align}
|y_1|   \asymp \textstyle \sqrt{y_2 \phantom{I\hskip -5.5pt}}  , 
\end{align}
and under these conditions there is a smooth function $W (  \varlambda  ; y_1, y_2 )$ with support in $ \varlambda  \asymp \sqrt[3]{ y_1y_2 \phantom{I\hskip -5.5pt}}  $, satisfying
\begin{align}\label{4eq: bounds for W}
\varlambda^{i}	\frac {\partial^{i} W (\varlambda ; y_1, y_2 ) } {\partial \varlambda^{i}} \Lt_{i} U^i ,
\end{align}
with the implied constants uniform in $y_1$ and $y_2$, such that 
\begin{align}\label{4eq: asymptotic of J}
	J (x; y_1, y_2) = \frac {e (- 3    \sqrt[3]{ xy_1y_2 \phantom{I\hskip -5.5pt}}  )} {  \sqrt[3]{ |xy_1y_2|}} W ( {\textstyle \sqrt[3]{  xy_1y_2 \phantom{I\hskip -5.5pt} }}; y_1, y_2) + O (T^{-A}).
\end{align}
\end{lem}

\begin{proof}
In view of \eqref{4eq: asymptotic, R+}, the integral $J (x; y_1, y_2)$ is exponentially small if $y_2 < - T^{\vepsilon}$ is negative, and it suffices to consider for $y_2 > T^{\vepsilon}$ integrals of the form
\begin{align}\label{4eq: J-integral}
  \frac 1 {  \sqrt[4]{ y_2 \phantom{I\hskip -5.5pt}} }	\int_0^{\infty} \varww^{\spm} (x_1, x; y_2) e \big(\hskip -1pt \pm 2 \sqrt{xy_2 / x_1 \hskip -1pt} -x_1 y_1 \big)   {\nd x_1} ,
\end{align}
for suitable smooth $ \varww^{\spm} ( x_1 , x; y_2)  $ supported in $x_1, x \asymp 1$, such that    \begin{equation}\label{4eq: bounds for w}
	\frac {\partial^{i+j } \varww^{\spm} ( x_1 , x; y_2) } {\partial x_1^{i} \partial x^j  }  \Lt_{i, j } U^{i+j } .  
\end{equation} 
The first assertion follows immediately from repeated partial integration (for example, one can use \cite[Lemma 7.1]{Qi-GL(3)}). Moreover, the integral in \eqref{4eq: J-integral} is also negligibly small if the $\pm$ does not match the sign of $y_1$.  As for the second assertion, we assume  for simplicity that $y_1 \asymp \ds \sqrt{y_2}$ is positive, let $\varlambda =   \sqrt[3]{  xy_1y_2 \phantom{I\hskip -5.5pt}}$, and make the change of variable $x_1 \ra   \sqrt[3]{xy_2/y_1^2} \cdot x_1 $ so that \eqref{4eq: J-integral} is transformed into
\begin{align}
	  \frac {e (- 3 \varlambda)} {\sqrt{|\varlambda|}} \int_0^{\infty} \varvv (x_1; \varlambda; y_1, y_2 ) e \big( \varlambda \big(3  - 2/ \sqrt{x_1} - x_1  \big) \big) \nd x_1, 
\end{align}
where $ \varvv (x_1; \varlambda; y_1, y_2 ) $ is smooth, supported in $ x_1 \asymp 1 $, such that
\begin{align}\label{4eq: bounds for v}
\varlambda^j	\frac {\partial^{i + j } \varvv ( x_1; \varlambda; y_1, y_2) } {\partial x_1^{i} \partial \varlambda^j  }  \Lt_{i, j } U^{i + j } . 
\end{align} 
Then \eqref{4eq: bounds for W} and \eqref{4eq: asymptotic of J} follow from Sogge's version of stationary phase estimates (see \cite[Theorem 1.1.1]{Sogge} and \cite[Lemma 7.3]{Qi-GL(3)}). 

It should be stressed that  the implied constants in  \eqref{4eq: bounds for w}  and \eqref{4eq: bounds for v} (also in  \eqref{4eq: bounds for W}) do not depend on $y_1$ or $y_2$. 
\end{proof}

The following lemma is essentially (8.16) in \cite{Young-Cubic} but stated  in the fashion of Proposition 11.3 and Corollary 11.10 in \cite{Qi-GL(3)} with necessary adaptions. It is obtained by the method of stationary phase.

\begin{lem}\label{lem: Phi + -}
Let $ y_2  > T^{\vepsilon} $ and $\varLambda \Gt T^2$.  Assume  $|y_1|   \asymp   \sqrt{y_2 \phantom{I\hskip -5.5pt}}$. Then for $y_1 y_2 \asymp Y$ we have
\begin{align}\label{5eq: w = Phi}
	e (\pm y_1 y_2/ \varLambda)  \widehat{\widetilde{\varww}}^{\spm} (y_1, y_2; \varLambda) = \frac {MT^{1+v}} {\sqrt{|y_1y_2|}} \Phi^{\spm} (y_1 y_2/\varLambda) + O (T^{-A}), 
\end{align}
such that {\rm(}for $x \asymp Y/\varLambda${\rm)} $\Phi^{\splus} (x) = 0$ and $\Phi^{\sminus}(x) = 0$  unless 
\begin{align}\label{5eq: range of Y, +}
	\sqrt[3]{|Y|}  \asymp \sqrt{\varLambda}, \qquad |Y|/\varLambda  > T M^{1-\vepsilon} , 
\end{align}
and
\begin{align}\label{5eq: range of Y, -}
	\sqrt{\varLambda} \asymp T , \qquad     \sqrt{\varLambda}/\sqrt[3]{|Y|}   >  M^{1-\vepsilon} , 
\end{align}
respectively, in which cases 
\begin{align}\label{5eq: Phi+=}
	\Phi^{\splus} (x) = \int e (Tr/\pi - x \tanh^2 r) V^+ (r) \nd r,
\end{align}
provided that $|x| < T^{2-\vepsilon}$,  and
\begin{align}\label{5eq: Phi-=}
	\Phi^{\sminus} (x) = \int e (Tr/\pi - x \coth^2 r) V^- (r) \nd r,
\end{align}
where $V^{\splus} (r)$, $V^{\sminus} (r)$ are
 supported in
 \begin{align}\label{4eq: support}
r \asymp T\varLambda /Y, \qquad |r| \asymp \sqrt[3]{|Y|}  / \sqrt{\varLambda} , 
 \end{align} 
respectively, satisfying $r^i (\nd/\nd r)^i V^{\spm} (r) \Lt_i U^i$.  
\end{lem}

Now we write $\Phi^{\spm}  (x) = \Phi^{\spm}_{T, M} (x) $ to indicate its dependence on $T$ and $M$. Moreover we add superscript like ${ \Phi^{\oldstylenums{0}\, \spm}_{T, M}} (x)$ or ${ \Phi^{\flat \, \spm}_{T, M}} (x)$ to indicate its origin from $\SDH^{\shskip \oldstylenums{0}}$ or $\SDH^{\shskip\flat}$. Lemma \ref{lem: Phi + -} remains valid if $\oldstylenums{0}$ or $\flat$ were attached in the notation. 

Let $M^{1+\vepsilon} \leqslant H \leqslant T/3$ and $|K - T| \leqslant H$. We would like to change   $T$  into $K$ and average $ K^{1+v}  \Phi^{\spm}_{K, M} (x) $ over $K$. However, it is only necessary to change   $T$ into $K$ in \eqref{5eq: w = Phi}, \eqref{5eq: Phi+=}, and \eqref{5eq: Phi-=}---the $T$ in \eqref{5eq: range of Y, +}, \eqref{5eq: range of Y, -}, and \eqref{4eq: support} are kept because $K \asymp T$. We stress that $V^{\oldstylenums{0}\, \spm} (r)$ is independent on $K$ as it is just a smooth truncation of $g^{\oldstylenums{0}} (r)$ according to  \eqref{4eq: support}, while   the definition of $ g^{\oldstylenums{0}} (r)$  as in \eqref{3eq: g0} does not involve $T$ (or $K$).

\begin{lem}\label{lem: average T} Let notation be as above. 
	Define
	\begin{align}\label{5eq: U +-, 0}
	U^{\splus} = T^2 \varLambda / |Y|, \qquad U^{\sminus} = T \sqrt[3]{|Y|}  / \sqrt{\varLambda}. 
	\end{align}
Then 
	\begin{equation}\label{5eq: average of Phi}
		\begin{split}
			\int_{T-H}^{T+H} \hskip -2pt K^{1+v}  \Phi^{\spm}_{K, M} (x)  \nd K \hskip -1pt = \hskip -1pt  \frac {T }  { U^{\spm} }  \big( \hskip -0.5pt (T\hskip -1pt + \hskip -1pt H)^{1+v}   \Phi^{\oldstylenums{1}  \, \spm}_{T + H, M} (x) \hskip -1pt - \hskip -1pt (T\hskip -1pt - \hskip -1pt H)^{1+v}  \Phi^{\oldstylenums{1}  \, \spm}_{T - H, M} (x) \hskip -0.5pt \big) & \\
		-  \frac {(1+v)T} {U^{\spm}} \hskip -1pt   \int_{T-H}^{T+H} \hskip -1pt K^v  \Phi^{\oldstylenums{1}  \, \spm}_{K, M} (x) \nd K \hskip -1pt +  \hskip -1pt  {M  U} \hskip -1pt  \int_{T-H}^{T+H} \hskip -1pt K^{ v} \Phi_{K, M}^{\flat \, \spm} (x) \nd K & ,   
		\end{split} 
	\end{equation}
for $ \Phi^{\oldstylenums{1}  \, \spm}_{K, M} (x) $ and $ \Phi^{\flat  \, \spm}_{K, M} (x) $ of the same shape as $\Phi^{\spm}_{K,M} (x)$. 
\end{lem}

\begin{proof}
Express $ \Phi^{\spm}_{K, M} (x) $ on the left of \eqref{5eq: average of Phi} as the sum $ \Phi^{\oldstylenums{0}   \, \spm}_{K, M} (x) + MU/K \cdot \Phi^{\flat \, \spm}_{K, M} (x) $ according to  \eqref{3eq: H + H0}. Then   \eqref{5eq: average of Phi} readily follows from 
the simple identity:
\begin{equation*}
	\begin{split}
		  \int_{T-H}^{T+H}  K^{1+v} e (Kr/\pi) \nd K =  &  \frac {K^{1+v} e (Kr/\pi)} {2   i r}\bigg| {\phantom{\int \hskip -7pt}}_{T-H}^{T+H}  - \frac {1+v} {2 i r} \int_{T-H}^{T+H}  K^{v} e (Kr/\pi) \nd K.
	\end{split} 
\end{equation*}
Note that the factor $1/r$ arises on the right, while, in view of \eqref{4eq: support}, $|1/ r| \asymp U^{\spm} / T$ on the support of $V^{\oldstylenums{0} \, \spm} (r)$, so we can define the weight function in $ \Phi^{\oldstylenums{1}  \, \spm}_{K, M} (x) $ to be  $V^{\oldstylenums{1} \, \spm} (r) = U^{\spm} / 2i T r \cdot  V^{\oldstylenums{0} \, \spm} (r)$.  
\end{proof}

\begin{rem}\label{rem: obstacle}
	It is substantial that the weight function $ V^{\oldstylenums{0}\, \spm} (r) $ is localized as in {\rm\eqref{4eq: support}}, since $r \ra 0$ is not allowed in view of the factor $1/r$ cause by the averaging process. By examining the analysis in {\rm\cite[\S 11]{Qi-GL(3)}}, unfortunately, we find that it is no longer the case if the field is not $\BQ$. 
\end{rem}

Finally, we return to the setting of Lemma \ref{lem: Phi + -} and record here the following expression of $\Phi^{\spm} (x)$ due to Young \cite[Lemma 8.3]{Young-Cubic} (see also \cite[Lemma 13.2]{Qi-GL(3)}). It is obtained by the technique of Mellin transform. 

\begin{lem}\label{lem: Phi after Mellin}
	Let $x \asymp X$. Suppose that $|X| > T^{\vepsilon}/\varLambda$ and $\varLambda \Gt T^2$. For  
	\begin{align}\label{condition +}
		 |X| \asymp \sqrt{\varLambda}, \qquad  T < \sqrt{\varLambda}/  M   ^{1-\vepsilon}, 
	\end{align}
or
	\begin{align}\label{condition -}
	T \asymp \sqrt{\varLambda}	, \qquad |X| < \sqrt{\varLambda} / M   ^{3-\vepsilon} , 
	\end{align}
in the $\pm$-case, 	respectively, we have 
	\begin{align}\label{eq: Phi pm}
		\Phi^{\spm} (  x) = \frac 1 {T} \int_{|t| \asymp U^{\hskip -0.5pt \spm}} \varlambda_{X, T}^{\spm} (t) |x|^{it} \nd t, 
	\end{align}
	with $\varlambda_{X, T}^{\spm} (t) \Lt 1$ {\rm(}$\varlambda_{X, T}^{\spm} (t)$ depends on $X$, $T$ but the implied constant does not{\rm)} and  
	\begin{align}\label{5eq: U +-}
		U^{\splus} = T^2/|X|, \qquad U^{\sminus} = \sqrt[3]{ |X| T^{2}}, 
	\end{align} 
provided that $|X| < T^{2-\vepsilon}$.
\end{lem}

Note that \eqref{condition +}, \eqref{condition -}, and \eqref{5eq: U +-} respectively are tantamount to \eqref{5eq: range of Y, +}, \eqref{5eq: range of Y, -}, \eqref{5eq: U +-, 0} on letting $ X = Y/ \varLambda$. Moreover, under the conditions in \eqref{condition +} or \eqref{condition -}, we have
\begin{align}\label{5eq: range of U}
	T^{\vepsilon} < U^{\spm} < \frac {T} {M^{1-\vepsilon}}, 
\end{align}
provided that $ T^{\vepsilon}/\varLambda < |X| < T^{2-\vepsilon} $.  

\begin{rem}\label{rem: no saving}
Note that the product of $1/T$ in {\rm\eqref{eq: Phi pm}} and   $ T/M^{1-\vepsilon} $ in {\rm\eqref{5eq: range of U}} equals $1 / M^{1-\vepsilon}$---it will appear to be  our saving in the error term for $\SM_3^{\snatural} (T, M)$.  Again, by  examining the analysis  in {\rm\cite[\S 13]{Qi-GL(3)}}, we find that there is no longer such a saving if the field is not $\BQ$. 
\end{rem}

\section{A Simple Lemma for the Hankel Transform}\label{sec: second Voronoi}

Our last analytic lemma is on the  Hankel transform of a special kind of functions that involve $x^{it}$.

\begin{lem}\label{lem: Hankel, it}
	Let $t, y > T^{\vepsilon}$. For fixed $\varvv (x) \in C_c^{\infty} (0, \infty)$  define
	\begin{align}
		\widetilde{\varvv}_{0} (0; t)   = & \int_0^{\infty}   {\varvv (x )}  \frac { \nd x} {x^{1/2-it}}  , \qquad \widetilde{\varvv\hskip 1pt}'_{\hskip -1pt 0}   (0; t) = \int_0^{\infty} {\varvv (x )} \log x \frac { \nd x} {x^{1/2-it}} , \\
		& \widetilde{\varvv}_{0} (\pm y; t) = y^{1/2+it}	\int_0^{\infty}   {\varvv (x )} B_0 ( \pm xy )  \frac { \nd x} {x^{1/2-it}} . 
	\end{align}
Then  $\widetilde{\varvv}_{0} (0; t)$, $ \widetilde{\varvv\hskip 1pt}'_{\hskip -1pt 0}   (0; t) $, and $\widetilde{\varvv}_{0} (- y; t)$ are all negligibly small, while  $ \widetilde{\varvv}_{0} (  y; t) $ is bounded but  negligibly small unless  $ y \asymp t^2 $. 
\end{lem}

\begin{proof}
The proof is standard. It is well-known that the Mellin integrals  $\widetilde{\varvv}_{0} (0; t)$, $ \widetilde{\varvv\hskip 1pt}'_{\hskip -1pt 0}   (0; t) $ are negligibly small.  In view of  \eqref{4eq: asymptotic, R+}, it is clear that $ \widetilde{\varvv}_{0} (- y; t) $ is exponentially small, while for $ \widetilde{\varvv}_{0} (  y; t) $ it is reduced to consider integrals of the form
\begin{align*}
	y^{1/4+it} \int_0^{\infty} \varww^{\spm} (x) e \lp (t/\pi) \log x \pm 2 \sqrt{y} x \rp \nd x, 
\end{align*}
for suitable $ \varww^{\spm} (x) \in C_c^{\infty} (0, \infty) $. By repeated partial integration (again, one can use \cite[Lemma 7.1]{Qi-GL(3)}), the integral is negligibly small unless  the sign is $+ $ and $ t   \asymp \hskip -2pt \sqrt{y}$, in which case, the second derivative test (\cite[Lemma 5.1.3]{Huxley}) may be applied to prove that the integral is bounded. 
\end{proof}

A direct consequence of     Lemma \ref{lem: Voronoi} (in the special case $c=1$) and Lemma \ref{lem: Hankel, it}  is the following truncated Vorono\"i summation formula. 

\begin{cor}\label{cor: dual}
Let $t, N > T^{\vepsilon}$. 	We have 
	\begin{align}
\sum_{n \sasymp N} \frac {\tau (n)}  {n^{1/2-it}} \varvv (n/N) = \sum_{m \sasymp t^2/N}  	\frac {\tau (m)}  {m^{1/2+it}} \widetilde{\varvv}_{0} (Nm; t) + O (T^{-A}), 
	\end{align}
for $ \varvv (x) $ fixed and $ \widetilde{\varvv}_{0} (y; t) $ bounded. 
\end{cor}

\section{Setup} \label{sec: setup}

As our motivation, we start with reviewing an asymptotic formula in \cite{Qi-Liu-Moments} for the twisted second moment of central $L$-values as follows:  
\begin{equation}\label{5eq: twisted second moment}
\begin{split}
	& \quad \ \sum_{f \in  \SB}  \omega_f \lambda_f (n_1) L \big( \tfrac 1 2 , f \big)^2 k  (t_f)  + \frac 1 {4\pi} \int_{-\infty}^{\infty}  \omega (t) \tau_{it} (n_1)    {\left| \zeta \big(\tfrac 1 2 + it \big) \right|^{4} }  
	k  (t) \shskip   \nd \shskip t \\
& =	 \frac { 4 \tau (n_1) M T} { \pi \sqrt{\pi   n_1}} \bigg(   \log \frac {T} { \sqrt{n_1}} + \gamma - \log 2\pi  \bigg) + O_{\vepsilon} \bigg( \bigg( \frac {M^3} {\sqrt{n_1}T} + \frac {\sqrt{n_1 T}  } {\sqrt{M} }\bigg) T^{\vepsilon}\bigg),  
\end{split}
\end{equation}
 for any $n_1 \leqslant T^{2-\vepsilon}$, while the second error term can be removed in the case $n_1 \leqslant M^{2-\vepsilon}$. This formula is proven by the ``Kuznetsov--Vorono\"i" approach, along with analysis for the Hankel and Mellin transforms of Bessel integrals. It should be stressed that the main term has two sources: half is from the diagonal term in Kuznetsov while the other half is from the zero frequency after Vorono\"i. 
 
 Heuristically, by summing up to $n_1 \leqslant T^{1+\vepsilon}$ according to the approximate functional equations (see \eqref{5eq: AFE, 1}, \eqref{5eq: AFE, 2}, and \eqref{2eq: derivatives for V(y, t), 1})\footnote{A subtle issue is that  the spectral $t_f$ and $t$ are involved in the weights $V_1 (n_1; t_f)$ and $V_1 (n_1; t)$, so the arguments and results in  \cite{Qi-Liu-Moments} can not be applied directly and  must be adapted slightly.}, albeit with a   weaker error term, we may already deduce from \eqref{5eq: twisted second moment} an asymptotic formula for the cubic moment of the form \eqref{1eq: smooth, asymptotic} in Theorem \ref{thm: T and M}. To strengthen the error term, in addition to the Vorono\"i summation used in  \cite{Qi-Liu-Moments}, we need to also apply   Poisson summation to the $n_1$-variable as Conrey and Iwaniec did in \cite{CI-Cubic}.


Now we turn to the smooth spectral cubic moment: 
\begin{align}
	\SM_3^{\snatural}  =	\sum_{f \in  \SB}  \omega_f L \big( \tfrac 1 2 , f \big)^3 k  (t_f)  + \frac 1 {4\pi} \int_{-\infty}^{\infty}  \omega (t)    {\left| \zeta \big(\tfrac 1 2 + it \big) \right|^{6} }  
	k  (t) \shskip   \nd \shskip t .
\end{align} For brevity, we suppress $T, M$ from our notation here, but keep in mind that we need to average the $T$-parameter later for the proof of Theorem \ref{thm: T and H}.

\delete{and
\begin{align}
	\SN_3^{ \shskip \flat }   =	\sum_{f \in  \SB}  \omega_f L \big( \tfrac 1 2 , f \big)^3 \varww  (t_f)  + \frac 1 {\pi} \int_{-\infty}^{\infty}  \omega (t)    {\left| \zeta \big(\tfrac 1 2 + it \big) \right|^{6} }  
	\varww (t) \shskip   \nd \shskip t ,
\end{align} 
with $k (t) = k_{T, M} (t)$ defined by \eqref{1eq: defn k(nu)}, and 
\begin{align}\label{5eq: defn w(nu)}
	 \varww  (t) = \varww_{T, M, H} (t) = \frac 1 { \sqrt{\pi}  M} \int_{\, T- H}^{T + H} k_{K, M} (t) \nd K  .
\end{align}}

By the    approximate functional equations \eqref{5eq: AFE, 1} and \eqref{5eq: AFE, 2},    we infer that
\begin{equation}\label{6eq: M, before Kzunetsov}
	\begin{split}
		\SM_3^{\snatural} = 4  \mathop{\sum \sum}_{ n_1, n_2 }  \frac {   \tau (n_2) } { \sqrt{n_1 n_2}  }  & \Bigg\{    \sum_{f \in  \SB  } \hskip -1pt   \omega_f  \lambda_f  ( n_1 ) \lambda_f  (n_2 ) V_1   ( n_1 ; t_f   )  V_2   ( n_2 ; t_f   ) k  ( t_f ) \\
& + \frac 1 {4 \pi} \int_{-\infty}^{\infty}  \omega (t)     
  \tau_{it}  ( n_1 ) \tau_{it}  (n_2 ) V_1   ( n_1 ; t   )  V_2   ( n_2 ; t ) k  ( t ) \shskip   \nd \shskip t	\Bigg\} .
	\end{split}
\end{equation}  
In view of  \eqref{2eq: derivatives for V(y, t), 1} in Lemma \ref{lem: afq}, at the cost of a negligible error term, we may truncate the $n_1$- and $n_2$-sums to the ranges $ n_1 \leqslant T^{  1 + \vepsilon}$ and $ n_2 \leqslant T^{  2 + \vepsilon}$ respectively. 

\section{Applying the Kuznetsov Trace Formula}
Next, we apply the Kuznetsov trace formula in Proposition \ref{lem: Kuznetsov} to the expression between the large brackets in \eqref{6eq: M, before Kzunetsov}, and then use \eqref{1eq: approx of V1} in Lemma \ref{lem: afq} with $U = \log T$ (so that the errors therein are negligible) to reformulate the resulting off-diagonal terms. More explicitly, we have
\begin{align}
	 \SM_3^{\snatural} = \SD_3 + \SO_3 + O (T^{-A}),  
\end{align}
where $\SD_3$ is the diagonal sum 
\begin{align}\label{6eq: diagonal}
	\SD_3 = 4 \sum_{n \leqslant T^{1+\vepsilon}} \frac {\tau (n)} {n} \SDH_3 (n), 
\end{align}
with
\begin{align}
\SDH_3 (n) = \frac {1} {2 \pi^2}	\int_{-\infty}^{\infty} V_1 (n; t) V_2 (n; t) k (t)    \tanh (\pi t) t  \shskip \nd \shskip t,  
\end{align}
while  $\SO_3$ is the off-diagonal contribution in the form
\begin{align}
	\SO_3 =  - \frac 1 {\pi^2} \int_{ \vepsilon - i U}^{\vepsilon + iU} \int_{ \vepsilon - i U}^{\vepsilon + iU} \SO_3 (v_1, v_2) \zeta (1+2v_2) e^{v_1^2 +2 v_2^2}  \frac {\nd v_1} {v_1} \frac {\nd v_2} {v_2}, 
\end{align}
with
\begin{align}\label{6eq: O3}
	 \SO_3 (v_1, v_2) = \sum_{\spm} \sum_{c }   \mathop{\mathop{\sum\sum}_{n_1 \leqslant T^{1+\vepsilon}}}_{ n_2 \leqslant T^{2+\vepsilon} }   \frac {\tau (n_2)} {n_1^{1/2+v_1} n_2^{1/2+v_2}} \frac{S(n_1, \pm n_2; c)} {c} \SDH \bigg( \hskip -1pt \pm \frac {n_1 n_2} {c^2}; v_1, v_2 \bigg), 
\end{align}
\begin{align}
	\SDH (x; v_1, v_2) = \frac 1 {2\pi^2} \int_{-\infty}^{\infty} h(t; v_1, v_2) B_{it} (x)   \tanh (\pi t) t  \shskip \nd \shskip t , 
\end{align}
and
\begin{align}
	h (t; v_1, v_2) = 
	k  (t) \delta (v_1, t) \delta (v_2, t)^2. 
\end{align}

It follows from \eqref{4eq: asymptotic for V2} that 
\begin{align}\label{6eq: H3, 2}
	\SDH_3 (n) = \frac {1} {2 \pi^2}	\int_{-\infty}^{\infty} V_1 (n; t) \lp \gamma  +  \psi_1 (t)  -  \log   \sqrt{n}   \rp k (t)    \tanh (\pi t) t  \shskip \nd \shskip t + O  (T^{-A}).   
\end{align}

By Lemma \ref{lem: x small} and \ref{lem: H small}, one may  impose the condition $       n_1 n_2     / c^2   \Gt  T^{2 } $ to the summations in \eqref{6eq: O3}, with the cost of a negligible error. 

\section{Applying the Vorono\"i and Poisson Summation Formulae} 
At this point, we introduce  smooth dyadic partitions for the $n_2$- and $n_1$-sums prior to the application of Vorono\"i and Poisson. More explicitly, we split $\SO_3 (v_1, v_2)$ into the sum of
\begin{equation*}
	\begin{split}
	\frac 1 {N_1^{1/2+v_1} N_2^{1/2+v_2}}    
 \sum_{\spm} \sum_{c} 
   \mathop{\sum\sum}_{ n_1, n_2 } \tau (n_2) \frac {S(n_1, \pm n_2; c)} {c} \varww^{\spm} \bigg( \frac {n_1} {N_1}, \frac {n_2} {N_2} , \frac {N_1 N_2} {c^2}; v_1, v_2\bigg), 
	\end{split}
\end{equation*}
for dyadic $1/2 < N_1 \leqslant T^{1+\vepsilon}$ and $1/2 < N_2 \leqslant T^{2+\vepsilon}$, where
\begin{align*}
	\varww^{\spm}  (  x_1, x_2   ;   \varLambda ; v_1, v_2 )  =   \varww (x_1; v_1) \varww (x_2; v_1) \SDH_2  ( \pm \varLambda x_1 x_2 ; v_1 ,  v_2  )    , \quad \varww (x; v ) = \frac{\varvv (x)   } {x^{1/2 + v }},
\end{align*}
for suitable $\varvv (x) \in C_0^{\infty}[1,2]$. For the moment, one may restrict the $c$-sum to the range $c \Lt \sqrt{N_1 N_2} / T$.

For the  $n_2$-sum,  we   open the Kloosterman sum $S ( n_1, \pm   n_2   ; c  )$ and apply the Vorono\"i summation formula as in Lemma \ref{lem: Voronoi}. %

For the entire zero-frequency contribution $\SZ_3$, after reversing the procedures of truncation and partition (of sums and integrals), the arguments in \cite[\S 12.2]{Qi-Liu-Moments} can be easily adapted to prove   
\begin{align}\label{7eq: Z3}
	\SZ_3 = 2 \sum_{n_1 \leqslant T^{1+\vepsilon}} \frac {1} {\sqrt{n_1 \hskip -1pt}} \SZ_3 (n_1) + O (T^{-A}), 
\end{align}
with 
\begin{align}\label{7eq: Z3, 2}
	\SZ_3 (n_1) = \frac 1 {2\pi^2} \frac {\tau (n_1)} {\sqrt{n_1  \hskip -1pt}} \int_{-\infty}^{\infty} V_1 (n_1; t) \lp 2 \gamma  + 2 \psi_1 (t)  -  \log  n_1   \rp k (t)    \tanh (\pi t) t  \shskip \nd \shskip t . 
\end{align}
We only remark that it is crucial to have the formula:
\begin{align*}
	\int_{-\infty}^{\infty}   B_{it} (x  ) |x|^{  s - 1}  { \nd x } = \frac{\gamma (s, t)}{\gamma (1-s, t)}. 
\end{align*}
By comparing \eqref{6eq: diagonal} and \eqref{6eq: H3, 2} with \eqref{7eq: Z3} and \eqref{7eq: Z3, 2}, it is clear that $\SD_3$ and $\SZ_3$ are equal to each other up to a negligible error. 

For the dual ($m_2$-)sum after Vorono\"i,   the exponential sum turns into  $S (n_1\mp m_2, 0; c)$. For the  $n_1$-sum, similarly,  we   open  the Ramanujan sum $S (n_1\mp m_2, 0; c)$ and apply the Poisson summation  formula as in Lemma \ref{lem: Poisson}. Now the dual exponential sum reduces to $e (\pm m_1 m_2/c)$ along with the condition $(m_1, c) = 1$.  Note that the zero frequency in the $m_1$-sum exists only when $c = 1$ and is negligibly small by Lemma \ref{lem: asymptotic of J}. 

It is left to consider the sum $\SS^{\spm} (N_1, N_2)$ defined by
\begin{equation}\label{8eq: S = }
N_1^{1/2-v_1} N_2^{1/2-v_2} \sum_{c > 0} \frac 1 {c^2} \hskip -1pt \mathop{\mathop{\sum\sum}_{m_1, m_2 \neq 0}}_{(m_1, c) = 1}   \tau (m_2) e \Big( \hskip -1pt \pm \frac {m_1m_2} {c} \Big) \widehat{\widetilde{\varww}}{}^{\spm} \hskip -1pt \bigg( \hskip -1pt \frac {m_1 N_1} {c}, \frac {m_2 N_2} {c^2} ; \frac {N_1 N_2} {c^2} \hskip -1pt \bigg),  
\end{equation}
where $ \widehat{\widetilde{\varww}}^{\spm} $ is the Fourier--Hankel transform defined by
\begin{align*}
	\widehat{\widetilde{\varww}}^{\spm} (y_1, y_2; \varLambda) = \int_0^{\infty} \int_0^{\infty} \varww^{\spm} (x_1, x_2 ; \varLambda ) e (-x_1 y_1) B_0 (x_2 y_2) \nd x_1 \nd x_2 , 
\end{align*}
for 
\begin{align*}
	\varww^{\spm}  (  x_1, x_2   ;   \varLambda  )  =   \varww_1 (x_1 ) \varww_2 (x_2 ) \SDH_2  ( \pm \varLambda x_1 x_2   )   
\end{align*}
with weight functions $\varww_1, \varww_2 \in C_c^{\infty} [1, 2]$ such that $ \varww_1^{(i)} (x ), \varww_2^{(i)} (x ) \Lt_{i} U^i  $ ($\varww_1 (x) = \varww(x; v_1)$ and $\varww_2 (x) = \varww(x; v_2)$). For brevity, we have suppressed $v_1$ and $v_2$ from our notation.

\section{Treating the Main Term}\label{sec: main term}
 
The purpose of this section is to prove that there is a certain cubic polynomial $P^{\snatural}_3 (X)$  such that
\begin{align}\label{8eq: main}
	 \SD_3 + \SZ_3 = \sqrt{\pi} MT P^{\snatural}_3 (\log T) + O (M^3 \log^2 T /T). 
\end{align}

In view of  \eqref{6eq: diagonal}, \eqref{6eq: H3, 2}, \eqref{7eq: Z3}, \eqref{7eq: Z3, 2}, along with \eqref{2eq: derivatives for V(y, t), 1}, \eqref{1eq: approx of V1}, and \eqref{3eq: defn k(nu)}, by completing the sum,  truncating the  integrals, and moving the sum inside the integrals,  it is reduced to consider 
\begin{align}\label{8eq: main integral}
\frac {8} {\pi^2}	\int_{T - M^{1+\vepsilon}}^{T + M^{1+\vepsilon}}   \lp   (  \gamma  +   \psi_1 (t)  ) \theta (t)    + \theta_1 (t) \rp   e^{- (t-T)^2/ M^2}    t  \shskip \nd \shskip t , 
\end{align}
where $\theta (t)$ and $\theta_1 (t)$ are the integrals
\begin{equation}\label{8eq: theta (t)}
\begin{split}
	\theta (t) & =	\frac 1 {2   \pi i } \int_{ \vepsilon - i U}^{\vepsilon + i U}  \delta (v, t)  \zeta (1+v)^2 e^{ v^2  }      \frac {\nd v} {v}, \\
\theta_1 (t) & =	\frac 1 {2   \pi i } \int_{ \vepsilon - i U}^{\vepsilon + i U}  \delta (v, t)  \zeta (1+v) \zeta' (1+v) e^{ v^2  }      \frac {\nd v} {v} . 
\end{split}
\end{equation}
Define   $\gamma$, $\gamma_1$, and $\gamma_2$ by 
\begin{align}
	\zeta (s) = \frac 1 {s-1} + \gamma - \gamma_1 (s-1) + \frac {\gamma_2} 2 (s-1)^2 + \cdots, \qquad s \ra 1,
\end{align}
so that
\begin{align}\label{8eq: zeta, 1}
	& \zeta (1+v)^2 = \frac 1 {v^2} + \frac {2 \gamma} {v} +  \gamma^2 -2 \gamma_1  + \cdots, \qquad \qquad \qquad \  v \ra 0, \\ \label{8eq: zeta, 2}
	& \zeta (1+v) \zeta' (1+v) = -\frac 1 {v^3} - \frac { \gamma} {v^2} - \gamma  \gamma_1   + \frac {\gamma_2} 2    +  \cdots, \qquad \shskip v \ra 0. 
\end{align}
Define $\psi_1 (t)$, $\psi_2(t)$, and $\psi_3 (t)$ by
\begin{align}\label{8eq: delta}
	\delta (v, t) = 1 + \psi_1 (t) v + \frac {\psi_2 (t)} 2  v^2 + \frac {\psi_3 (t)} 6  v^3 + \cdots, \qquad \ \, \,  v \ra 0, 
\end{align} 
By the Stirling  formula for the derivatives of $\log \Gamma (s)$, we have
\begin{align}\label{8eq: psi}
	\psi_{k} (t) =    \lp \log \RC (t) - \log (2 \pi) \rp^k + O \big(  \log^{k-1} \RC (t)  / \RC(t)^2 \big). 
\end{align}

Now we shift the integral contour in \eqref{8eq: theta (t)} further down to $\mathrm{Re} (v) = - A$ and calculate the residues at $ v = 0 $ with the aid of  \eqref{8eq: zeta, 1}--\eqref{8eq: psi}. It follows that the integral in \eqref{8eq: main integral} turns into
\begin{align}\label{8eq: integral, 2}
	\frac {8} {\pi^2}	\int_{T - M^{1+\vepsilon}}^{T + M^{1+\vepsilon}}   S_3  (\log \RC(t) - \log 2\pi)  e^{- (t-T)^2/ M^2}    t  \shskip \nd \shskip t + O (M \log^2 T / T), 
\end{align}
where  $S_3 (X)$ is defined by
\begin{align}\label{9eq: defn of S}
	S_3 (X) = \frac 1 3 X^3 + 2 \gamma X^2 + (3\gamma^2 - 2 \gamma_1) X + \gamma^3 - 3 \gamma_1 +\frac   {\gamma_2} 2. 
\end{align} 
Finally, by the change of variable $t \ra T+Mt$, one can easily deduce  \eqref{8eq: main}  with
\begin{align}\label{9eq: defn of Pn}
	 P^{\snatural}_3 (X) = \frac {8} {\pi^{2}} S_3 (X - \log 2\pi). 
\end{align}

\section{Applying the Second Vorono\"i and the Large Sieve} \label{sec: apply large sieve}

In this section, we use the analysis of Young as in \S \ref{sec: analysis Hankel}, the   Vorono\"i as in \S \ref{sec: second Voronoi},  and the large sieve of Gallagher as in \S \ref{sec: large sieve} to estimate the sum $ \SS^{\spm} (N_1, N_2)$ defined by \eqref{8eq: S = }.

By Lemma \ref{lem: asymptotic of J} and \ref{lem: Phi + -}, we infer that, up to a negligible error,  
\begin{equation}\label{10eq: S +- =}
 \SS^{\spm} (N_1, N_2) =	\frac {M T^{1+v}} {N_1^{v_1} N_2^{v_2}}  \sum_{c \shskip > 0} \frac 1 {\sqrt{c}} \hskip -1pt \mathop{\mathop{\sum\sum}_{|m_1|, m_2  > 0}}_{(m_1, c) = 1}   \frac{\tau (m_2)} {\sqrt{|m_1 m_2|}} \Phi^{\spm} \lp \frac {m_1 m_2} {c} \rp ,   
\end{equation}
where the summations are subject to the following conditions:
\begin{equation}
	|m_1| N_1  \asymp \sqrt{m_2 N_2},  
\end{equation}
and
\begin{align}
	c < \frac {\sqrt{N_1 N_2}} {M^{1-\vepsilon} T }, \qquad |m_1m_2| \asymp \sqrt{N_1 N_2},  
\end{align}
in the $+$ case, or
\begin{align}
	c \asymp \frac {\sqrt{N_1 N_2}} {  T }, \qquad |m_1m_2| <\frac {\sqrt{N_1 N_2}} {M^{3-\vepsilon}},  
\end{align}
in the $-$ case. In particular, it follows    that $$y_2 = \frac{ m_2 N_2}   {c^2} \geqslant \frac {  N_2} {c^2} \Gt \frac {T^2} {N_1} \geqslant T^{1-\vepsilon}, $$ and $$|x| = \frac { |m_1m_2|} c \leqslant |m_1m_2| \Lt \sqrt{N_1 N_2} \leqslant T^{3/2+\vepsilon}, $$ 
by   $N_1  \leqslant T^{1+\vepsilon}$ and $N_2 \leqslant T^{2+\vepsilon}$, so the assumptions in Lemma   \ref{lem: Phi + -} are satisfied. 
Moreover, recall that $v = v_1 + 2 v_2$ and $\mathrm{Re}(v_1) = \mathrm{Re}(v_2) = \vepsilon$. 

Next, we use the M\"obius function to relax the condition $(c, m_1) = 1$, and then introduce  dyadic partitions to the variables $c$, $m_1$, and $m_2$, it follows that if we define
\begin{align}
\overline{S_{it} (C)} = \sum_{c \sim C} \frac {1} {c^{1/2+it}}, \qquad	S_{it} (L) = \sum_{m \sim L} \frac {1} {m^{1/2-it}}, \qquad S_{it}^{\snatural} (L) = \sum_{m \sim L} \frac {\varvv (m)} {m^{1/2-it}}, 
\end{align}
for suitable $\varvv \in C_c^{\infty} [1,2]$, 
 then Lemma \ref{lem: Phi after Mellin} implies that $\SS^{\spm} (N_1, N_2)$ is bounded by the supremum of 
\begin{equation}\label{almostthere}
 \ST^{\spm} (C^{\spm}, L_1^{\spm}, L_2^{\spm}) =	M      T^{\vepsilon}   \int_{|t| \asymp U^{\spm}} \big|\overline{S_{it} (C^{\spm})} S_{it} (L_1^{\spm}) S_{it}^{\snatural} (L_2^{\spm}) \big| 
 \nd t  ,
\end{equation}
for dyadic parameters $C^{\spm}$, $L_1^{\spm}$, and $L_2^{\spm}$ in the ranges
\begin{align}\label{10eq: range +}
	& C^{\splus} < \frac {\sqrt{N_1N_2}} {M^{1-\vepsilon}  T}, \qquad L^{\splus}_1 \Lt   \frac {\sqrt{L^{\splus}_2 N_2}} {N_1}, \qquad    L^{\splus}_2 \asymp N_1, \\
\label{10eq: range -}	& C^{\sminus} \Lt \frac {\sqrt{N_1 N_2}} {T}, \qquad L^{\sminus}_1 \Lt \frac {\sqrt{L^{\sminus}_2 N_2}} {N_1}, \qquad L^{\sminus}_2 < \frac {N_1} {M^{2-\vepsilon}}, 
\end{align}
and for
\begin{align}\label{10eq: U +-}
	U^{\splus} = \frac { C^{\splus} T^2}{L^{\splus}_1 L^{\splus}_2}, \qquad U^{\sminus} = \frac { \sqrt[3]{L^{\sminus}_1 L^{\sminus}_2 T^2} } {\sqrt[3]{C^{\sminus}}}. 
\end{align}
Recall from \eqref{5eq: range of U} that \begin{align}\label{8eq: range of U}
	T^{\vepsilon} < U^{\spm} < \frac {T} {M^{1-\vepsilon}}. 
\end{align}  

In view of Corollary \ref{cor: dual}, for the $m_2$-sum of length $L_2^{\spm}$, its dual sum is of length $U^{\spm\, 2} / L_2^{\spm}$, so we can always ensure that the length of summation does not exceed $ U^{\spm}$. For the $c$-sum and the $m_1$-sum, we take the complex conjugate of the former and group them together so that the new sum is over $c m_1$ and of length $C^{\spm} L^{\spm}_1$.  Furthermore, it follows from \eqref{10eq: range +}, \eqref{10eq: range -}, and \eqref{10eq: U +-}, along with $N_2 \leqslant T^{2+\vepsilon}$, that
\begin{align*}
	C^{\splus} L^{\splus}_1 \Lt \frac {C^{\splus} L^{\splus}_2 N_2} {L^{\splus}_1 N_1^2} \Lt \frac {C^{\splus}   N_2} {L^{\splus}_1 L^{\splus}_2} \Lt	U^{\splus} T^{\vepsilon}, 
\end{align*}
and
\begin{align*}
	C^{\sminus} L^{\sminus}_1 \Lt \frac {C^{\sminus} \sqrt[3]{L^{\sminus}_1 L^{\sminus}_2 N_2}} {\sqrt[3]{N_1^2}} \Lt 
	\frac {  \sqrt[3]{L^{\sminus}_1 L^{\sminus}_2} N_2} {\sqrt[3]{C^{\sminus}T^4}} \Lt	U^{\sminus} T^{\vepsilon}.
\end{align*}
We conclude by Cauchy--Schwarz and Lemma \ref{lem: large sieve} that 
\begin{align}\label{10eq: bound for T, 1}
	\ST^{\spm} (C^{\spm}, L_1^{\spm}, L_2^{\spm}) \Lt M U^{\spm} T^{\vepsilon}, 
\end{align}
and hence by \eqref{8eq: range of U} that
\begin{align}\label{10eq: bound for T, 2}
	\ST^{\spm} (C^{\spm}, L_1^{\spm}, L_2^{\spm}) \Lt  T^{1+\vepsilon}. 
\end{align}

\section{Conclusion}

Firstly, Theorem \ref{thm: T and M} follows immediately from \eqref{8eq: main} and \eqref{10eq: bound for T, 2}. 

Next, we prove Theorem \ref{thm: T and H}. To this end, we invoke \eqref{1eq: M = int M} in Lemma  \ref{lem: average}, change $T$ into $K$, and average \eqref{8eq: main} and \eqref{10eq: S +- =} over $K$ from $T-H$ to $T+H$. Clearly, \eqref{8eq: main} yields the main term in \eqref{1eq: asymptotic} along with an error $ O (M^2 H \log^2 T / T) = O (M H T^{\vepsilon}) $. As for \eqref{10eq: S +- =}, after the dyadic partitions, we use the expression \eqref{5eq: average of Phi} in Lemma \ref{lem: average T}  for $$ \int_{T-H}^{T+H} K^{1+v} \Phi^{\spm}_{K, M} (m_1m_2/c) \nd K. $$
It follows from \eqref{10eq: bound for T, 1} and \eqref{10eq: bound for T, 2}  that the contributions from the terms with $\Phi^{\oldstylenums{1}  \, \spm}$ and $\Phi^{\flat  \, \spm}$ on the right of \eqref{5eq: average of Phi} are bounded by $O (T^{1+\vepsilon} + H T^{\vepsilon}) = O (T^{1+\vepsilon})$  and $O (M H T^{\vepsilon})$ respectively; the cancellation of $U^{\spm}$ is crucial here.  We conclude that 
\begin{align}
	\frac 1 {\sqrt{\pi} M} \int_{T-H}^{T+H} \SM_3^{\snatural} (K, M) \nd K = & \int_{T-H}^{T+H}  K P^{\snatural}_3 (\log K) \nd K   + O  (  T^{1+\vepsilon} + M H T^{\vepsilon} ) .  
\end{align}
Also note that there is an error term $ O (M T^{1+\vepsilon})$  in \eqref{1eq: M = int M} in Lemma  \ref{lem: average}. 
Consequently, we obtain \eqref{1eq: asymptotic} on choosing $M = T^{\vepsilon}$. 

Finally,  Corollary \ref{cor: Ivic} (Ivi\'c's strong moment conjecture) follows from Theorem \ref{thm: T and H} if we choose  $H=T/3$    and apply a dyadic summation. It is clear that $P_3$ and $P^{\snatural}_3$ are related by 
\begin{align}
	\frac {\nd \lp K^2 P_3(\log K) \rp} {\nd K} = K P^{\snatural}_3 (\log K). 
\end{align}
	
	
	\newcommand{\etalchar}[1]{$^{#1}$}

\end{document}